\documentclass{amsproc}
\usepackage{graphicx,amsmath,amssymb,fixmath,color}

\newtheorem{theorem}{Theorem}[section]
\newtheorem{lemma}[theorem]{Lemma}

\newtheorem{corollary}[theorem]{Corollary}

\theoremstyle{definition}

\theoremstyle{remark}

\numberwithin{equation}{section}

\begin{document}
\large

\title{ New perspectives: stability and  complex dynamics  of food webs via
Hamiltonian methods}



\author{Vladimir Kozlov}
\address{Dept.of Mathematics, University of Linkoping, 58183, Linkoping, Sweden}

\author{Sergey Vakulenko}
\address{Institute for Mechanical Engineering Problems, Russian Academy of Sciences  Saint Petersburg, Russia
and  \\ Saint Petersburg National Research University of Information Technologies, Mechanics and Optics.
Saint Petersburg, Russia
}

\author{Uno Wennergren}
\address{Dept.of Ecology, University of Linkoping, 58183, Linkoping, Sweden}

\begin{abstract}{ We investigate  global stability and dynamics of large   ecological networks
by  classical methods of the dynamical system theory, including  Hamiltonian  methods, and averaging.  Our analysis
exploits the network
topological structure, namely, existence of strongly connected nodes (hubs) in the networks. We reveal new relations between topology,
interaction structure  and network dynamics. We describe mechanisms of
 catastrophic  phenomena  leading to  sharp changes of dynamics and
  investigate how these phenomena
 depend on ecological interaction structure. We show that
a Hamiltonian structure of interaction leads to stability and large biodiversity.
   }
\end{abstract}

\maketitle



\date{25/08/2012}


\section{Introduction}\label{sec1}

In this paper, we consider a global dynamics of large ecological
networks with complex topology.
The last decade, the topological
structure  of  biological networks and, in particular, ecological
networks (food webs) has been received a great attention (see
\cite{AB, Bas1, Bas2, Bas3,  Dunne2,
Krause}). Different  indices
were introduced and studied in detail for real models and random
assembled networks (see for an overview \cite{Dorman}). These
indices are connectance, cluster coefficients, degree distribution,
number of compartments, and many others \cite{Dorman}. They reflect
important topological properties of networks, for example, the
degree distribution show that the ecological networks contain a few
number of strongly connected nodes, while the compartment number
describes
 their decomposition in compartments \cite{Krause}.  The networks can contain  different substructures (for example, when a guild of species
 contains  specialists with few links and generalists with many links).
 Many works have investigated a connection between the network structure and  fragility (see, for example, \cite{Montoya,
Pimm}).  Great efforts has been done  to reveal connections between the
network topological structure  and their robustness and fragility
\cite{Alles,  Bas1, Bas3,  Dorman}. In
particular, it is found that
 mutualistic networks show interesting effects, for example, nestedness \cite{Bas1,  Bas3}
 and a truncation in degree distribution \cite{Bas3}.

One of key problems is to find a qualitative description of dynamics generated by a  network of  a given topological structure. This problem
is connected with the local and global stability of networks.  Until May's seminal works \cite{May1, May2}, ecologists believed
 that huge complex ecosystems,  involving
 a larger number of species and connections,
are more stable \cite{Hut}.  May \cite{May1, May2} considered
 a community of $S$ species with connectance
$C$ that measures the number of realized links with respect to  the
number of possible links. By investigation of the matrix  defining
stability of a local equilibrium, it was shown that the instability
can increase with respect to $C$. More connected communities would
be more unstable. These ideas  were developed in \cite{Alles}, where
networks consisting of predator-prey modules  were studied. It is
shown that if predator-prey interactions are prevalent, then the
complex community is stable. This local approach was developed in
\cite{Alles1, Alles2}, where more complicated networks with
interactions of different types (predator-prey, amensalism,
mutualism, competition) were studied. The technique, used in   works
\cite{May1,May2, Alles, Alles1, Alles2}, allows us to  study only
local stability of equilibria. Important results on global stability
are obtained in \cite{Hofbauer}, where the Lyapunov function method
 is applied (see Theorem 15.3.1).  In this case all trajectories tend
to a unique equilibrium.

However, real ecological systems exhibit both a  complicated behaviour
with complex transient dynamics and stability. We observe here
a chaos, periodical oscillations, bursts and transient phenomena,
connected with multistability.
 But up to now, there is no
analytic approach to describe ecological dynamics including such effects
as  chaos, multistability,   ecological
bursts  and global stability of ecological networks supporting
different interactions and having complicated topology.

	We consider  a  class  of the  Lotka-Volterra
systems consisting of two groups of species with interactions between these groups only. This is a natural generalization of the classical prey-predator model and under some natural
assumptions such systems are reduced to Hamiltonian ones (for brevity, we refer to them as
HLV systems).
Note that some Lotka-Volterra systems admit
a  special Hamiltonian representation involving a skew-product Poisson matrix depending on species-abundances \cite{Plank} that allows to find a constant of motion.

In this paper, we first introduce  Hamiltonian representations of the HLV  based on   canonical variables. Such canonical structure
 simplifies an analysis of dynamics and sharply facilitates the
use of classical perturbation methods.
Moreover, we can obtain a very  short description of complex systems. Consider,
for example,
so-called "star-system", when one of the groups consists of one species (a predator species feeding on a number of  prey-species), presents  examples of HLV systems.
 Such star structures often appear in ecological webs.
Neglecting species concurrence and self-limitation effects, we obtain that the star subsystem can be described by   simple Hamiltonians. In our canonical variables,
these ecological Hamiltonians correspond to  nonlinear oscillators, which can be described by two canonical variables only and  well
studied in physics and mechanics. We show, by this analogy, that this dynamics is integrable
and exhibits interesting phenomena, for example, kink and soliton solutions.
We find a dependence of solution types on ecological interactions involved in the system.
An interesting effect, which does not exist in mechanics, but arises in ecology,
is a domino effect. Namely,
in star systems, it is possible  that extinction of a corner-stone species leads to extinction of all species. We find the conditions on interactions, which makes it possible.
Note that, from a dynamical point of view, solitons correspond to  so-called homoclinic
curves \cite{Robinson}.  This implies an important consequence. Namely,
if we consider a weakly perturbed star-system, this homoclinic structure can generate
a chaos, an effect well studied in many works for  mechanical and physical applications.
Another interesting effect is that in ecology, in contrast with
mechanics and physics, the Hamiltonians of star systems involve some positive
constants, which are defined by initial data. This means that the star system dynamics has a ''memory".  The solution form and period  depend on these constants.

The case of large random systems is particularly interesting. The papers
  \cite{May1, May2} are focused on the stability of local
	equilibria in such systems, in works \cite{Alles1, Alles2} the structure of ecological
	interactions is taken into account. In  \cite{Williams} a niche model is introduced
	and it was shown, by computer simulations, that this remarkably simple model
	predicts key structural properties of the complex food webs. Work \cite{Gross}
	investigated the stability of large random networks assembled by the niche model.
	
	By our sufficient and necessary conditions of persistence, we consider
	global stability of random  HLV systems with interactions of predator-prey type.
		Note that HLV structure is consistent with the niche model.
	There occur two systems of linear equations
	which must have positive solutions. Unfortunately, the probability that both systems
		have such solutions  tends to zero as the size system $N \to ´+\infty$.
		However, one can suppose that predator-species and prey-species can use
		an adaptive strategy. Mathematically, it can be described if we assume that one can vary some coefficients, which define ecological interactions. It can be interpreted as an {\em adaptive strategy}.  Moreover, we assume that networks involves a fixed number of
		species-generalists. This  allows us to show that
		 the probability
that  these conditions are fulfilled tends to $1$ as $N \to ´+\infty$.
This analytical result confirms conclusions \cite{Gross}, where food webs stability is  studied  for more general random models.    Numerical results of  \cite{Gross}
show that food web stability is enhanced when species at a high trophic level feed
on multiple prey species or species at an intermediate trophic level are fed upon multiple
predator species. This means that species-generalists increase stability.
Note that they play a key role in our approach since potential energy
in Hamiltonian is defined via species-generalist abundances.

We examine the notions of stability  and  structural stability
  for large ecological systems. We demonstrate that  some restrictions should be imposed
on the signs of coefficients and on the structure of the interaction matrix
in order to have a stable system.
For  HLV systems we prove the following  results on global stability.
First, we state sufficient and necessary conditions for strong persistence, i.e., ecological
stability. We demonstrate  that,
under some natural assumptions, the  HLV systems with concurrence and self-limitations
 are permanent even if both these effects are small. In this case
the corresponding ecological systems have a large  biodiversity, and they are globally stable nonetheless they can exhibit a complicated transient dynamics.
Such   systems may be permanent and structurally unstable. In a permanent  system we can observe complicated transient dynamics.


In order to study these perturbation effects in more detail, we use an
asymptotic approach, which allows us to investigate  dynamics of
weakly perturbed HLV systems.
This approach is based on  methods,  developed in the Hamiltonian system
theory and mechanics \cite{Arnold,  Robinson, Kam, Sari}.   The canonical structure, that are found  by us for HLV, simplifies
	the application of these methods.



We consider different types of perturbations.
They can be generated by a varying environment, for example, by climate changes.
Another type of perturbations is due to a topological structure.      The key point, that
allows us to find a simplified description of the dynamics and stability, is existence of
species-generalists and a random scale-free topology of the
networks. As an example, let us consider a predator species feeding on a number
of species prey.  They form a star subsystem in the food web. If we
consider this star subsystem as a separate unit (niche), we can study
dynamics of this subsystem in the case of
 a weak self-limitation.  The important
observation is that in random scale-free topology different star subsystems ( niches) are weakly overlapped, therefore, such a system can be viewed as
a union of almost independent  niches.  Therefore,  the corresponding
dynamics can be described by averaging methods \cite{Arnold,  Sari}. We conclude that
there are possible different interesting effects such as existence of bursts, chaos and quasiperiodic solutions, and resonances. These effects appear if different niches
interact (overlap),  and even a weak interaction can lead to chaos or resonances.
Resonances can provoke an instability. The weak resonance effect can be repressed by
a self-limitation.

The advantage of the proposed approach is clarified when we
analyze an environment influence or study transient dynamics.
in the
Hamiltonian case it is sufficient to investigate a
''potential energy". The form  of this energy
 depend on the network topology and
interconnection forces. An analysis of the energy form allows us to find possible kinds of
transient dynamics in the system, i.e., what can happen in this
network: oscillations,  bursts, or sharp transitions, and also to
check  the stability. This analysis is particularly simple
for star systems.

In the case of slow environment variations, we derive an equation, which
describe a slow evolution of the energy. This equation allows
us to analyze possible effects. We show that the energy evolution can lead
to a sharp change of solution form, for example, there are possible
transitions from periodical solutions to  solitons, and vice versa.

We show that in weakly perturbed
HLV systems resonances are possible. Note that work \cite{King} considered
the case of two species  predator-prey systems perturbed  by small time periodic climate
variations.  We consider internal resonances, when
the resonance effect arises as a result of interactions inside the networks, without
external influences.
The mechanism of such  resonance effect
is as follows.  Let us consider a network consisting  of many disjoint
star subunits without self-limitation and concurrence. Such system lies in the HLV class.
Really existing networks contain star subsystems, which can interact (for example,
two different predator-species share some prey-species). Assume that this
interaction is weak. Then, to describe the network dynamics,
 we can apply standard perturbation methods.
We show that, under some assumptions on the system parameter,
 this interaction may generate a resonance effect, which
leads to an exponential instability of equilibrium. Such catastrophic instability is possible   if prey-predator interactions are mixed with other ones, which perturb prey-predator system. Note that the case of resonances induced by a periodic
oscillations of system parameter for two species is considered in  \cite{King}.

One of key questions of ecology is a connection between complexity
and stability \cite{May1}.  We consider networks
of random structure (with random ecological interactions). A choice
of random interactions in these networks follows works \cite{May1,
Alles, Alles1}. Numerical simulations and some analytical
considerations show that scale-free networks of large size can be
stable if a  parameter analogous to the May one \cite{May1} is less
than  a critical value.  We reveal a mechanism of  destruction  of
large networks in the case when the May parameter is large enough.
This mechanism is connected with a network topology and formation of
hubs of large connectivities that can lead to an "overlapping" phenomenon
 (see sections \ref{topol}  and \ref{random}).

The most of results are analytic  but we also use numerical
simulations to illustrate analytical results and to study networks
of random structure.


\section{Organization of the paper}\label{sec2}

In Sect.\ref{sec3}, we formulate  the Lotka-Volterra model with a
self-limitation describing interaction of two species communities (for
example, plants and pollinators, or preys and predators).

The investigation of  ecological stability  is based
on a reduction of the problem to a Hamiltonian system, which is presented
 in next sections
\ref{trans} and \ref{Ham}.  Namely,
 we transform equations to another system
and obtain, under some natural assumptions, a Hamiltonian formulation
of these equations.  Furthermore, in Sect. \ref{globs}
we formulate and prove main results on global stability of HLV systems.
In Sect. \ref{pred} star systems are considered.
   Sect. \ref{pred2}  concerns with  the case of varying environment, which is modeled by an ecological system consisting of one generalist species and several specialist species  with
   coefficients depending on time. We include in the model a weak self-limitation also.

A resonance  analysis for ecological system consisting of two generalist species and several specialist species is performed in Sect. \ref{Reson}.

We deal with the stability-complexity problem for large random
networks in  section \ref{random}.




\section{Statement of problem}\label{sec3}

\subsection{Topology of networks} \label{topol}

We consider the following  Lotka-Volterra system describing an interaction between two groups of species
$x$ and $v$:
\begin{equation}
     \frac{dx_i}{dt}=x_i (-r_i  + \sum_{k=1}^M a_{ik} v_k -  \ \sum_{j=1}^N \gamma_{ij} x_j),
    \label{LVX}
     \end{equation}
\begin{equation}
     \frac{dv_j}{dt}=v_j (\bar r_j   -  \sum_{l=1}^N b_{jl} x_l  - \sum_{k=1}^M d_{jk} v_k),
    \label{LVV}
     \end{equation}
     where $i=1,..., N$, $j=1,..., M$ and
      $N+M$ is the total number of species with abundances
     $x_i$ and $v_j$.
The coefficients $r_i$  and $\bar r_j$ are  intrinsic growth (or
decay) rates for species $x_i$ and $v_j$, respectively. The matrices
${\bf A}$ and $\bf B$ with the entries $a_{ij}$ and $b_{ij}$,
respectively,  determine an interaction between two groups of
species, whereas the matrices $\bf \Gamma$ with entries
$\gamma_{ij}$ and $\bf D$ with entries $d_{ij}$ correspond to
self-limitation effects. This model describes an ecological  system
with two trophic levels.

The topological structure of the networks consists of a
directed graph $(V,  E)$, where $V$ is a set of vertices and $E$ is
a set of edges (links).  We distinguish two types of nodes,
 $V_1=\{1,2,..., N\}$ and  $V_2=
 \{1,2,..., M\}$.  Thus, $V=V_1 \cup V_2$.

 The edge $e=\{i, j\}$ belongs to $E$ if one of the alternatives is fulfilled:  a) $a_{ij} \ne 0$ when $i\in V_1$ and $j\in V_2$; b) $\gamma_{ij} \ne 0$ when $i\in V_1$ and $j\in V_1$;
 c) $b_{ij} \ne 0$ when $i\in V_2$ and $j\in V_2$; d) $d_{ij} \ne 0$ when $i\in V_2$ and $j\in V_2$.


 Connectance ${\mathbb C}$ is an important characteristic of the network and it
   is defined as the number of the ecological links divided by the number of all possible links:
 \begin{equation} \label{Conn}
 {\mathbb C}=2|E|/(N+M)(N+ M -1),
 \end{equation}
where $|E|$ is the number of edges.

In the scale-free networks  the degree distribution of a node is
\begin{equation} \label{scalefree}
Pr_k= C k^{-s}
\end{equation}
(see \cite{AB}), where $Pr_k$ is the probability for a node to have
$k$ adjacent nodes and the exponent $s$ lies within the interval
$(2,3)$. The networks with such  property usually have a low number
of strongly connected nodes (hubs) whereas the remaining ones are
weakly connected.  In our case this means that we have several
species-generalists and many species-specialists. Each generalist
(hub) is a center of a ''star subsystem" consisting of many species. We
study the dynamics of such subunits in Sect. \ref{pred}.

Some species can correspond to  nodes adjacent to two different hubs.
This means that two star subsystems are overlapped, or, in biological
terms, two different predators are feeding on the same prey.  Numerical simulations,
where scale-free networks were generated by the standard preferential attachment
algorithm, show that  this overlapping is small, the number
of nodes sharing two different centers $<< N+M$.


\subsection{Dynamics} \label{dyn}

 We consider  system (\ref{LVX}), (\ref{LVV})
 in the positive cone ${\mathbb R}^{N+M}_{>}=\{x=(x_1,...,x_N), v=(v_1,...,v_M): \ x_i > 0, v_j >0 \}$. This cone is invariant
 under dynamics (\ref{LVX}), (\ref{LVV}) and we assume that
  initial data  always lie in this cone:
\begin{equation}
x(0)=\phi \in  {\mathbb R}^N_{>}, \quad v(0)=\psi \in  {\mathbb R}^M_{>}.  \label{ID}
     \end{equation}
 We distinguish the following main cases:

{\bf PP} (predator-prey). If   $v_j$ are  preys  and $x_i$ are predators, then 
\begin{equation}
       a_{il}  \ge 0, \quad b_{jk} \ge 0, \quad r_i > 0, \quad \bar r_j > 0;
    \label{signs1}
     \end{equation}

{\bf MF} (facultative mutualism)
\begin{equation}
      a_{il} \ge 0, \quad b_{jk} \le 0, \quad r_i < 0  \quad\bar r_j > 0;
    \label{signs2}
     \end{equation}

{\bf MO} (obligatory mutualism)
\begin{equation}
      a_{il} \ge 0, \quad b_{jk} \le 0, \quad  r_i > 0 \quad \bar r_j < 0;
    \label{signs3}
     \end{equation}
and

{\bf C} (competition)
\begin{equation}
      a_{il} \le 0, \quad b_{jk} \ge 0, \quad r_i > 0, \quad \bar r_j > 0.
    \label{signs4}
     \end{equation}

Note that, if  $a_{il}  \le 0, \ b_{jk} \le 0, \quad r_i < 0, \quad \bar r_j < 0$,
then we are dealing with  the {\bf PP} case, where $v_j$ are  predators  and $x_i$ are preys.

Systems, where we observe a species-generalist and a number of species-specialists (for example, $M=1$ and
$N >> 1$, or $N=1$ and $M >> 1$) are omnipresented as important structural elements  in ecological networks.
In this case the  topology of interconnections  involves the so-called  ''star structure" \cite{Bas3}.

We have a pure star structure if all products $a_i b_i$  are of the same sign.
The case  $a_i b_i >0$ corresponds to an
{\bf M} - star  structure, and the case $ a_i b_i < 0$ corresponds to a {\bf P}- star  structure.
Dynamics of these star networks is
quite different.  We also consider mixed structures, where $a_i b_i$ may have different signs.

In subsequent sections we  describe the dynamics of  large
ecological networks.
 These
networks are considered as sets consisting of star structures.  This
approach connects contemporary ideas in network topology and
classical methods of dynamical system theory. First, we investigate
dynamical properties of a single star system,  and  then, using weak
overlapping property, we develop a perturbation approach for weakly
interacting star systems.    Result is a new description of dynamics, which essentially extend
  results on local stability of large networks
studied in \cite{May1, May2, Alles, Alles1}. It allows us to reveal different mechanisms of catastrophic phenomena in networks. We describe periodical dynamics, bursts, chaos,
and resonances.

\section{Transformation of Lotka-Volterra system}
\label{trans}

To study oscillations in the star systems (where $M \le N$),  we
make a transformation of equations (\ref{LVX}), (\ref{LVV}) to
another system with respect to variables $q=(q_1,\ldots,q_M)\in\mathbb R^M$
and $C=(C_1,\ldots,C_N)\in\mathbb R_{>}^N$, which is defined as
follows:
\begin{equation}
     \frac{dC_i}{dt}=  V_i(C, q),
    \label{LVXC}
     \end{equation}
\begin{equation}
     \frac{d^2 q_j}{dt^2}=\Big (\frac{dq_j}{dt} + \mu_j\Big )\Big ( F_j(C, q)
     - \sum_{k=1}^M d_{jk}\Big ( \frac{dq_k}{dt} +\mu_k\Big )\Big ),
    \label{LVQ}
     \end{equation}
     where $i=1,..., N$, $j=1,..., M$,
\begin{equation} \label{Vi}
     V_i(C, q)=C_i (\bar \gamma_i  -  \sum_{k=1}^N \gamma_{ik} C_k \exp({\bf A}_k\cdot q)),
\end{equation}
\begin{equation} \label{VFi}
F_j(C, q)=\bar r_j   - \sum_{k=1}^N b_{jk} C_k \exp({\bf A}_k\cdot q),
\end{equation}
and
\begin{equation} \label{GB}
\bar\gamma_i=-r_i +\sum_{m=1}^M a_{im}\mu_m.
\end{equation}
Here $\mu_1,\ldots,\mu_M$ are  positive constants and ${\bf
A}_i\cdot q=\sum_{k=1}^M a_{ik} q_k$. One can verify  the following
assertion.

\begin{lemma}
\label{Lham}{Let $q$ and $C$ be a solution to {\rm (\ref{LVXC})} and {\rm (\ref{LVQ})} with initial data
$$
q_k(0)=\alpha_k,\;q'_k(0)=\beta_k,\;k=1,\ldots,M\;\;\mbox{and}\;\;C_i(0)=C_i^0,\;i=1,\ldots,N,
$$
where $\beta_k>-\mu_k$, $k=1,\ldots,M$. Then the functions
\begin{equation}
     x_i(t)=C_i \exp(\sum_{k=1}^M a_{ik} q_k(t)), \quad
     v_k=\frac{dq_k}{dt} + \mu_k,
    \label{XR}
\end{equation}
solve system {\rm (\ref{LVX})}, {\rm (\ref{LVV})} with the initial conditions
$$
x_i(0)=C_i^0\exp(\sum_{k=1}^M a_{ik} \alpha_k), 
$$
Moreover, all solutions to system {\rm (\ref{LVX})}, {\rm (\ref{LVV})} can be obtain by solving  {\rm (\ref{LVXC})} and {\rm
(\ref{LVQ})} with appropriate initial conditions.}
\end{lemma}

System (\ref{LVXC}) and (\ref{LVQ}) can be reduced to a first order system if we introduce  the new variables  $p_j$ by
$$
\frac{dq_j}{dt}+ \mu_j=\exp(p_j),\;j=1,\ldots,M.
$$
We obtain then
\begin{equation}
     \frac{dq_j}{dt}=\exp(p_j) -\mu_j
    \label{EQ}
     \end{equation}
     and
\begin{equation}
     \frac{dp_j}{dt}=F_j(C, q)  -  \sum_{l=1}^M d_{jl} \exp(p_l),
    \label{EP}
     \end{equation}
where $j=1,..., M$ .

\section{Hamiltonian} \label{Ham}

\subsection{Reduction to a Hamiltonian system}\label{red}

Equation (\ref{LVXC}) takes a particularly  simple form when
$\gamma_{ij}=0$ and $\bar \gamma_i=0$. Then   the right-hand side in
(\ref{LVXC})
 equals zero and hence $C_i$ is a constant. Therefore, if $p$ and $q$  solve system
(\ref{EQ}), (\ref{EP}) supplied with the initial conditions
$$
q_k(0)=\alpha_k\;\,\mbox{and}\;\; p_k(0)= \log(\beta_k + \mu_k),\;k=1,\ldots,M,
$$
then the corresponding solution to (\ref{LVX}), (\ref{LVV}) is given by  (\ref{XR}).

We assume additionally that $d_{jl}=0$. Then system
(\ref{EQ}), (\ref{EP}) can be rewritten as a Hamiltonian system
provided the matrices ${\bf A}$ and ${\bf B}$ satisfy the relations
\begin{equation}\label{K1a}
\sigma_l b_{lk}=\rho_k a_{kl},\;\;k=1,\ldots,N,\;\;l=1,\ldots,M,
\end{equation}
where $\rho_l$ and $\sigma_l \ne 0 $ are real numbers \cite
{Plank}(for biological interpretation of this condition see the remark at the end of this section). Indeed, let
$$
\tilde p_j=p_j\;\;\mbox{and}\;\; \tilde q_j=\sigma_j q_j.
$$
Then relations (\ref{EQ}) and (\ref{EP}) imply
\begin{equation}\label{K1b}
\frac{d \tilde p_j}{dt}=F_j(C, \tilde q)
\end{equation}
and
\begin{equation}\label{K1c}
\frac{d \tilde q_j}{dt}=\sigma_j(\exp(\tilde p_j) -\mu_j),\;\;j=1,\ldots,M.
\end{equation}
We introduce two functions
$$
\Phi(C, \tilde q)= \sum_{k=1}^N \rho_k C_k \exp(\sum_{l=1}^Ma_{kl}\sigma_l^{-1}\tilde q_l)-\sum_{k=1}^M\bar r_k \tilde q_k
$$
and
$$
\Psi(\tilde p)=\sum_{k=1}^M\sigma_k(\exp(\tilde p_k)-\mu_k \tilde p_k).
$$
One can verify  that
$$
\frac{\partial \Phi(C, \tilde q)}{\partial \tilde q_j}=-F_j,\;\;\mbox{}\;\;\frac{\partial \Psi(\tilde p)}{\partial \tilde
p_j}=\sigma_j(\exp(\tilde p_j) -\mu_j).
$$
Thus, system (\ref{K1b}), (\ref{K1c}) takes the form
\begin{equation}\label{K2b}
\frac{d\tilde p_j}{dt}=-\frac{\partial H(C, \tilde p, \tilde q)}{\partial \tilde q_j}
\end{equation}
and
\begin{equation}\label{K2c}
\frac{d\tilde q_j}{dt}=\frac{\partial H(C, \tilde p, \tilde q)}{\partial \tilde p_j},\;\;j=1,\ldots,M,
\end{equation}
where
\begin{equation}\label{K2a}
H(C, \tilde p, \tilde q)=\Phi(C, \tilde q)+\Psi(\tilde p).
\end{equation}

{\bf Remark} \label{Rem1}.  Relation (\ref{K1a}) admits a biological interpretation. Consider, for example, a predator-prey system.  Then
 condition (\ref{K1a}) means that the coefficients
$a_{kl}$ and $b_{lk}$ are proportional to
the frequency of meetings between $k$-th predator and $l$-th prey, when the predator-species
is feeding on  the prey-species. Note that if $M=1$ or $N=1$, this condition is fulfilled.
It corresponds to the case of a star  structure.

\section{Global stability of  dynamics}  \label{globs}

\subsection{Strong persistence and permanence}

Let us remind definitions of permanency and strong persistence. The general Lotka-Volterra system
\begin{equation}
     \frac{dy_i}{dt}=y_i(-  R_i  + \sum_{k=1}^N W_{ik} y_{k}), \quad i=1,..., N,
    \label{LVG}
     \end{equation}
     is said to be permanent if there exist $\delta >0$
and $D >0$ independent  of the initial data such that
\begin{equation}
     \lim \inf_{t \to +\infty}y_i(t) \ge \delta,
    \label{perm}
     \end{equation}
	\begin{equation}
     \lim \sup_{t \to +\infty}y_i(t) \le D
    \label{perm2}
     \end{equation}			
for every solution to (\ref{LVG}) (see \cite{Hofbauer}).
The system is strongly persistent, if $\delta$ and $D$ in (\ref{perm}) and
 (\ref{perm2})  may depend  on  initial data.

The strong persistence property
means that the system is ecologically stable and all species coexist. System (\ref{LVG})
can be strongly persistent
only if the corresponding linear system
\begin{equation}
WY=R
\label{eq1}
     \end{equation}
has a positive solution
(i.e., all $Y\in\Bbb R_>^N$)\cite{Hofbauer}.
Here $W$ is the matrix with the entries  $w_{ij}$, and $R$, $Y$ are vectors
with components $R_l, Y_m$, respectively.

Let us present some necessary and sufficient conditions of boundedness of
trajectories of system (\ref{K2b}), (\ref{K2c}) (we omit the sign of
tilde  to simplify notation). The trajectories $q(t)$ are bounded
under the following conditions: $\sigma_n >0, \quad n=1,..., M$, and
\begin{equation}
\lim \ \Phi(C, q) = +\infty \quad as \ |q| \to +\infty, \quad  q \in {\mathbb R}^N.
\label{stab25}
\end{equation}
In the next assertion we present some conditions, which guarantee
 the asymptotic property (\ref{stab25}).

\begin{theorem} \label{T1} Assume that
\begin{equation} \label{bound173}
\rho_k>0,\;\,k=1,\ldots,N.
\end{equation}
Then (\ref{stab25}) is equivalent to to the following:

\noindent
 the rank of matrix $\{b_{jl} \}$ is $M$ and there exists a  vector $z=(z_1,\ldots,z_N) \in
{\mathbb R}_>^N$
  such that
\begin{equation} \label{bound1}
\bar r_l=\sum_{j=1}^N b_{lj}z_j,\;\;  \quad l=1,\ldots,M.
\end{equation}
\end{theorem}
\begin{proof}

Let ${\mathbb K}$ be the closed convex set
$$
{\mathbb K}=\{\eta \in {\bf R}^M:  \eta_l= \sigma_l^{-1}
\sum_{k=1}^N a_{kl} z_k,  \  z_l \ge 0, l=1,...,N \}.
$$
Since all numbers $\rho_k$ and $C_k$ are positive, the property ({i}) is
equivalent to
$$
q \in {\mathbb  K}^* \implies \bar r \cdot q <0,
$$
where
$$
{\mathbb  K}^* =\{ \xi \in {\mathbb R}^M : \xi\cdot \eta \le 0 \quad \mbox{for \ all }\eta \in {\mathbb K} \}.
$$
The last property can be also formulated as $\bar r$ belongs to the
interior of $ (({\mathbb K})^*)^*$. Using that $(({\mathbb
K})^*)^*={\mathbb K}$, we obtain (\ref{stab25}) is equivalent to $\bar r \in $ the interior
of ${\mathbb K}$, which is exactly the assertion of Theorem due to (\ref{K1a}).

\end{proof}



  The conditions $\bar \gamma_i=0$, $i=1,...,N$ and (\ref{bound1})
means that  system (\ref{LVX}), (\ref{LVV})  has a positive equilibrium in ${\mathbb R}_>^{N+M}$. So, we obtain

\begin{corollary} \label{Corol}
{Let $d_{kl}= \gamma_{ij}=0$ in (\ref{LVX}), (\ref{LVV})
and let relation (\ref{K1a}) be fulfilled with positive
$\rho_k$ and $\sigma_l$. Then system (\ref{LVX}), (\ref{LVV}) is
strongly persistent if and only if the rank of the matrix $A$ is $M$
and algebraic systems
\begin{equation}
      \sum_{k=1}^M a_{ik} v_k =r_i,
    \label{LVX4}
     \end{equation}
\begin{equation}
       \sum_{l=1}^N b_{jl} x_l=\bar r_j
    \label{LVX4N}
     \end{equation}		
have  positive solutions.
}
\end{corollary}

The following result means that 
when there is a  self-limitation,
condition (\ref{K1a}) guarantees  global stability of  ecological equilibria.

The next theorem deals with a general Lotka-Volterra system
\begin{equation}\label{LVXd}
     \frac{dx_i}{dt}=x_i (-r_i  + \sum_{j=1}^M (a_{ij}+A_{ij}) v_j -  \ \sum_{l=1}^N \gamma_{il} x_l),
     \end{equation}
\begin{equation}\label{LVVd}
     \frac{dv_j}{dt}=v_j (\bar r_j   -  \sum_{l=1}^N (b_{ji}+B_{ji}) x_i  - \sum_{k=1}^M d_{jk} v_k),
     \end{equation}
     where, as before, $i=1,..., N$ and $j=1,..., M$. We denote by ${\bf a}$, ${\bf A}$, ${\boldsymbol{\gamma}}$, $\textbf{b}$, ${\bf B}$ and ${\bf d}$ the matrices 
     with entries  $a_{ij}$, $A_{ij}$, $\gamma_{il}$, $b_{ji}$, $B_{ji}$ and $d_{jk}$ respectively and introduce the block-matrix
\begin{equation*}
    \mathfrak{M} =
     \left (\begin{array}{ll}
{\boldsymbol{\gamma}}&-{\bf A} \\
{\bf B}&{\bf d}
\end{array}
\right )\left (
\begin{array}{ll}
{\boldsymbol{\rho}}^{-1}&0 \\
0&{\boldsymbol{\sigma}}^{-1}
\end{array}
\right ),
\end{equation*}
where ${\boldsymbol{\rho}}^{-1}$ and ${\boldsymbol{\sigma}}^{-1}$ are diagonal matrices with the entries $\rho_1^{-1},\ldots,\rho_N^{-1}$ and $\sigma_1^{-1},\ldots,\sigma_M^{-1}$ on the diagonal respectively.
\begin{theorem}\label{T4s} We assume that system {\rm (\ref{LVXd})}, {\rm (\ref{LVVd})} has a positive equilibrium. Let the matrices ${\bf a}$ and ${\bf b}$ satisfy condition {\rm (\ref{K1a})} with positive $\rho_i$ and $\sigma_j$. If the matrix $\mathfrak{M}$ is positive definite, i.e.
$$
(\xi,\eta)\mathfrak{M}(\xi,\eta)^T>0
$$
for all $\xi\in\Bbb R^N$ and $\eta\in\Bbb R^M$ such that $|\xi|+|\eta|\neq 0$ then system {\rm (\ref{LVXd})}, {\rm (\ref{LVVd})} is permanent.

\end{theorem}


{\bf Proof}.
First let us make a change of variables $X_i=\rho_i x_i$, $i=1,\ldots,N$, and $X_{N+j}=\sigma_j v_j$, $j=1,\ldots,M$. 
Then we obtain the system 
\begin{equation}\label{Syst3}
\frac{dX_m}{dt}=X_m(-R_m-\sum_{k=1}^{N+M}(\mathfrak{A}_{mk}+\mathfrak{M}_{mk})X_k),\;\;m=1,\ldots,N+M,
\end{equation}
where $R=(r_1,\ldots,r_N,\bar{r}_1,\ldots,\bar{r}_M)$ and
\begin{equation*}
    \mathfrak{A} =
     \left (\begin{array}{ll}
0&-{\bf a} \\
{\bf b}&0
\end{array}
\right )\left (
\begin{array}{ll}
{\boldsymbol{\rho}}^{-1}&0 \\
0&{\boldsymbol{\sigma}}^{-1}
\end{array}
\right ).
\end{equation*}
System (\ref{Syst3}) has a positive equilibrium and it is permanent if the matrix $\mathfrak{A}+\mathfrak{M}$ is positive definite. the last is equivalent to positive definiteness of the matrix
$\mathfrak{M}$ due to (\ref{K1a}).

According to  Theorem  \ref{T4s}
  even small
self-limitation and concurrence  can stabilize a system with
a Hamiltonian structure.  In this case an elementary analysis
(see subsection \ref{hamxv}) shows  that, in an equilibrium state,   all the species coexist.  This means that ecological systems with weakly perturbed
Hamiltonian structure  can have large biodiversity. If the Hamiltonian condition
 (\ref{K1a}) is violated then for small $\lambda_g$ and $\lambda_d$ the competition exclusion principle shows that only a single species can survive.

Let us consider the stability for perturbed Hamiltonian systems. Note
that these perturbations can be connected with more complicated interaction topology
and violations of condition (\ref{K1a}). In the Hamiltonian case we have the square interaction matrix
$W=W_H$  of the size $N+M$, which can be  decomposed in $4$ blocks
$$
\begin{bmatrix}
    0       & A  \\
    B       &  0  \\
\end{bmatrix}
$$
that corresponds to $2$ trophic levels. Moreover, matrices
$A$ and $B$ are connected via condition (\ref{K1a}). Let us consider
the perturbed interaction matrix $W=W_H + \tilde W$.  Let us define the vector
$d$  with $N+M$ components by
$
d=(\rho_1, ..., \rho_{N},   \sigma_1, ..., \sigma_M).
$


Consider the matrix    $\tilde W^{(d)}$ with the entries
$$
\tilde W^{(d)}_{kl}= d_k d_l^{-1} \tilde  W_{kl}, \quad k,l=1,...M+N.
$$

If the  matrix $\tilde W^{(d)}$  is negatively defined  and for  system  (\ref{LVG})  there exists a positive equilibrium, then system (\ref{LVG})  is  permanent.
 It can be shown by the same arguments as in the proof of Theorem
\ref{T4s}.

Theorems \ref{T1} and  \ref{T4s} state the results on global stability.
 In real situations, many ecological networks
are unstable: catastrophic phenomena are possible. In Sect. \ref{pred} and \ref{random}
 we study mechanisms of such phenomena  by the Hamiltonian methods.

\subsection{Stability for large number of species} \label{largeN}

We start this section by showing that the
Lotka-Volterra system (\ref{LVG}) of a random structure  has no positive equilibria.
Consider the set  of all $N \times N$ matrices $A$ with entries
$a_{ij}$ uniformly bounded by a constant
$$
| A_{ij}| < K,
$$
such that each row of $A$ is non-zero and contains at most $M_r$ non-zero entries and
 each column is also non-zero and contains at most $M_c$ non-zero entries.
Using the standard Lebesgue measure $\mu$ defined on ${\mathbb M}_{K, N}$, for any measurable
$C$  we introduce the probability $P(C)=\mu(C)/\mu({\mathbb M}_{K, N})$.

\begin{theorem}\label{T3}   {  Let $B$ be a  vector with $N$ components and
$A \in {\mathbb M}_{K, N}$.
Then the probability $P_N$ that the linear equation $AY=B$ has a positive solution $Y$
tends to zero as $N \to +\infty$}.
\end{theorem}

{\bf Proof}. Let  $Y=(y_1,..., y_N)$ be a positive solution of (\ref{eq1}).
We can assume that all
row and columns of $A$ contain
at least non-zero entries
(probability to have a matrix with a zero row or column is $0$).
Let us change a sign of $k$-th row in $A$ that gives a matrix $A^{(k)}$.  Equation $A^{(k)} Y=B$ has the solution $Y= (x_1,...,-x_k,..., x_N)$, which is not positive. So, each matrix $A$, for which $AY=B$ has a positive solution, corresponds at least $N$ different matrices, for which these solutions are not positive.
			Therefore, $P_N \le 1/N$.
			
This result admits an ecological interpretation. Consider a
random large ecological network. Assume that the connectance of this network
is bounded. If we have no restrictions on ecological interactions,
such network has no positive stationary states and thus it is not ecologically stable with a probability close to
$1$. A possible variant of such a restriction can be a sign restrictions on the coefficients of the system or
condition (\ref{K1a}), which leads to a Hamiltonian structure.  

The next example demonstrates that the notion of structural stability for large Lotka-Volterra system must be used with discretion.

{\bf Example.} Consider the general Lotka-Volterra system (\ref{LVG}) where $A$ is the identity matrix. Then the matrix $A$ has the eigenvalue $1$ only. Consider the matrix
$A_\varepsilon=A-\varepsilon B$ where $B$ is the $N\times N$ matrix with all elements equal $1$. If we take the sum of all rows we get that it is equal to $1-N\varepsilon$. Therefore, if $\varepsilon=1/N$ the matrix $A_\varepsilon$ has zero eigenvalue.

The following theorem says that the conditions proved in Theorem \ref{T1} which are equivalent to (\ref{stab25}) are satisfied with probability close to $1$ for large $N$.

\begin{theorem}\label{T2}
{Let  $M$ be fixed and let
$\bar r_j$, $j=1,...,M$, be random numbers mutually independent and normally
distributed according to standard normal law ${\bf N}(r_0,
\sigma^2)$,  where $\sigma \ne 0$ and $r_0 >0$. Let also
coefficients $b_{jk}$, $j=1,\ldots,M$ and $k= 1, \ldots, N$, be
mutually independent random numbers subjected to the normal law
${\bf N}(0, 1)$.
 Then  condition {\rm (\ref{bound1})} is fulfilled and the matrix $B$ has rank $M$ with probability $1- \epsilon_N$,
where $\epsilon_N\to 0$  as $N \to +\infty$.}
\end{theorem}

\begin{proof} Let $B_j=\{b_{1j}, ..., b_{Mj} \}$ and $R=(r_1,..., r_M)$,  $j=1,..., N$. To prove Theorem it is sufficient to show that
the vector $R$ belongs to the convex cone, which coincides with all linear combinations with positive coefficients of $M$  vectors from the set $\{B_j\}_{j=1}^N$
and these vectors are linear independent with probability $\geq 1-\epsilon_n$, where $\epsilon_n\to 0$ as $N\to\infty$.

 We identify vectors $B_j$ and $R$ with points
$B_j/|B_j|$ and $R/|R|$ respectively on the sphere
$$
S_M=\{w: w=(w_1,..., w_M):  \ |w|=(w_1^2 + ... + w_M^2)^{1/2}=1 \}.
$$
Let us introduce the sets
$$
S_m^\pm(\epsilon)=\{ w \in S_M:  \ |w\pm e^m| < \epsilon \}, \;\;m=1,\ldots,M,
$$
where  $e^m$ the unit vectors with components $e_{k}^m=\delta_{k}^m$,  \ $k=1,\ldots, M$. Let also
$$
S(\epsilon)=\{ w \in S_M:  \ |w_k \pm 1| > 2\epsilon, \quad k=1,..., M \}.
$$

One can check the following properties:

(1)  Probability that the number $R/|R|$ lies  in $S(\epsilon)$ can be estimated from below by $1- C\epsilon$, where $C > 0$ is a
constant;

(2)   Probability that at least one of vectors $B_j/|B_j|$, $j=1,..., N$ belongs  to $S_m^\pm(\epsilon)$, $m=1,\ldots,M$, can be estimated from below
by
$$
1- M \Big(1- \frac{|S_m(\epsilon)|}{|S_M|}\Big)^{N-1},
$$
where $|S|$ is the measure of $S$;

(3)  If $S_m^\pm(\epsilon)$, $m=1,\ldots,M$, contains at least one vector $B_j$ and the vector $R$ belongs to $S(\epsilon)$ then it is inside the convex cone of certain $M$ vectors from different
$S_m^\pm(\epsilon)$.

 These
properties prove the theorem.
\end{proof}

Consider some biological corollaries and interpretations
of these results,  in particular,  Theorems \ref{T1}, \ref{T2}
and corollary \ref{Corol}.  Mathematically, persistence follows
from existence of positive solutions
of   systems  (\ref{LVX4}) and (\ref{LVX4N}). Let $M << N$, and all
interactions are random.  Theorem \ref{T2} asserts
that then the second system has a solution with a probability close to $1$.
However, the same arguments, as in the proof this theorem,  show then that
  system  (\ref{LVX4})  has a positive solution with probability close to $0$.

	To overcome this difficulty and understand origins of large system
	stability,  one can suppose that real ecological systems can use
		{\em an adaptive strategy}.  Indeed, let us remind that fundamental
		relation (\ref{K1a}) admits an interpretation by meeting frequencies
 between predators and preys (see Remark \ref{Rem1}). This
frequencies are defined by the coefficients $\sigma_k$ and $\rho_i$.
Using this fact, we rewrite
(\ref{LVX4}) as follows:
\begin{equation}
      \sum_{k=1}^M \sigma_k b_{ki} v_k = \rho_i r_i, \quad i=1,...,N.
    \label{LVX4A}
     \end{equation}
For fixed $\sigma_k$ and $\rho_i$ this system has a solution with a small
 probability for large $N$.  However, let us suppose that predator-species and prey-species can change the meeting frequency (i.e., adjust $\sigma_k$ and $\rho_i$) (this means, biologically, existence of adaptive behaviour).  Then these coefficients become unknowns
and now (\ref{LVX4A}) always has a solution if we assume that the signs of coefficients are preserved under their random choice.

Finally,
Theorem \ref{T2} shows that in the Hamiltonian case
 the dynamical stability can be reinforced by an increase of $N$
and  an adaptive  strategy.

\section{Star structures}
\label{pred}

In the case of the Hamiltonian stucture we can develop
a general approach to instability.  If a Hamiltonian system  a single positive equilibrium,
then all level sets $H(p,q)=E$ have the same topological structure.
 Catastophical phenomena appear if topology of these level sets changes
depending on  $E$.

We start with the simplest case,
when we are dealing with a star structure.

\subsection{Star structures without self-limitation} \label{star1}

As in the previous section,  we assume that $\gamma_{ij}=0,
 d_{jl}=0$ in system (\ref{EQ}), (\ref{EP}) and condition (\ref{GB}) holds. These assumptions guarantee,
in particular, that $C_i$ are constants.
Let $M=1$, i.e., we deal with a star structure. Then  condition (\ref{GB}) becomes
\begin{equation} \label{GB10}
r_i = a_{i1}\mu_1, \quad i=1,..., N
\end{equation}
for some positive $\mu_1$.
We set $a_i=a_{i1}$,
$b_j=b_{j1}$ and $q=q_1$, $p=p_1$. Let us denote $\bar r=\bar r_1$.
System (\ref{EQ}), (\ref{EP}) takes the form
\begin{equation}
     \frac{dq}{dt}=\exp(p) -\mu_1, \quad  \frac{dp}{dt}=f(C, q),
    \label{QH}
     \end{equation}
where
 \begin{equation}
  f(C, q)=-\sum_{j=1}^N b_j C_j \exp(a_j q) + \bar r.
    \label{fQ}
     \end{equation}
     This is a Hamiltonian system with the
   Hamiltonian
\begin{equation}
     H(q, p)=\Psi(p)+\Phi(q),
    \label{EQH}
     \end{equation}
where
\begin{equation}
     \Psi(p)=\exp(p) -\mu_1 p,
    \label{KQ}
     \end{equation}
   is a ''kinetic energy",
and
\begin{equation}
  \Phi(C, q)=\sum_{j=1}^N \rho_j C_j \exp(a_j q) -  \bar r q
    \label{phiQ}
     \end{equation}
     is a ''potential" energy. Here $\rho_j=b_j/a_j$ and $\sigma_j=1$.
 The  function $\Psi$ is convex, goes to $+\infty$ as
$|p| \to \infty$,  and has a  minimum, which is $\mu_1(1- \log \mu_1)$ and is attained at  $p=\log \mu_1$.

System  (\ref{QH}) has the energy integral
\begin{equation}
     \Psi(p)+\Phi(C, q)=E=const.
    \label{energy00}
     \end{equation}
Proceeding as in \cite{Whitham}, we can describe solution of (\ref{QH}) in terms of the
function $\Phi(C, q)$  and the energy level $E$. The
values of $q$ satisfying  (\ref{energy00}) lie in the set
$$
D(E)=\{q:   \quad \Phi(q) \le E -\min \Psi(p)=E - \mu_1(1- \log \mu_1)  \}.
$$
This set is a union of intervals, which can be bounded or unbounded.  The ends of these intervals are defined by
\begin{equation}
     \Phi(C, q)=E-\mu_1(1- \log \mu_1).
    \label{level}
\end{equation}
After finding $q$ the component $p$ can be reconstructed  from (\ref{energy00}).


\subsection{Oscillations, solitons and kinks} \label{Osk}

Let us consider some important typical situations.   In this section,  $a_i$ and $b_i$  are arbitrary.

({\bf i}) If  equation (\ref{level}) has a unique root or it has no roots, then the corresponding interval is infinite, and we have
non-periodic solutions $(q(t), p(t))$, which are unbounded in $q$. Then  some of $x_i(t)$ go to $0$ or $+\infty$ as $t \to \infty$ and the
original system is not ecologically stable;

({\bf ii})  If   (\ref{level}) has two non-degenerate roots $q_- < q_+$ and $\Phi(C, q) < E- \mu_1(1- \log \mu_1)$ for all $q \in
(q_{-}, q_{+})$ then  $(q(t, C, E), p(t, C, E))$  is a periodic solution of the amplitude $A=q_{+} -q_{-}$  (see Fig. \ref{Fig3}).  The
period $T$ is defined by
\begin{equation}
   T=  \int_0^T dt = \int_{q_-}^{q_+} \Big(\frac{dq}{dt}\Big)^{-1} dq=\int_{q_-}^{q_+} (\exp(p(q))-\mu_1)^{-1} dq,
    \label{period}
\end{equation}
where  $p(q)$ can be found from (\ref{energy00}). The period $T$ depends on $E$ and $C$. For example,
 we have only  periodic solutions for the  {\bf PP } case (see Fig. \ref{Fig1}).

({\bf iii}) If  $\Phi(C, q)$ has a  local maximum at $q=q_{+}$, which is $E - \mu_1(1- \log \mu_1)$, and the second root in (\ref{level})
is non-degenerate, (see Fig. \ref{Fig3}),
 we obtain a soliton.
   Its graph  has a local burst in time,
 and  $q(t) \to q_{+}$ as $t \to \infty$.

({\bf iv})
 The kink solution corresponds to the case when $\Phi$ has two local maxima
at $q_{\pm}$ such that
 $\Phi(q_{-})=\Phi(q_{+})=E- \mu_1(1- \log \mu_1)$.  The kink describes a  jump in $t$ that can correspond to
 a sharp change of ecological behaviour.

 The kink solutions are unstable under a small perturbation of  $\Phi(C, q)$, whereas solitons are stable
 under such perturbations.
 When the parameter $E$ changes, we
 observe  a transition (via solitons or kinks) between different periodic solutions and
transitions from a periodic solution to unbounded in time solution and vice versa. Solitons and kinks appear only if we have a  star
system with different signs $a_i b_i$ (for example, a combination of {\bf PP} and {\bf C}).

\begin{figure}[h!] \label{Fig1}
\vskip-0.5truecm
 \includegraphics[width=80mm]{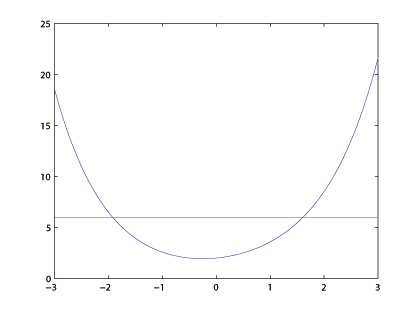}
\caption{\small The graph of  the potential $\Phi$ for case  {\bf PP}  and  small values of $\bar r$}
\end{figure}

\begin{figure}[h!]  \label{Fig2}
\vskip-0.5truecm
 \includegraphics[width=80mm]{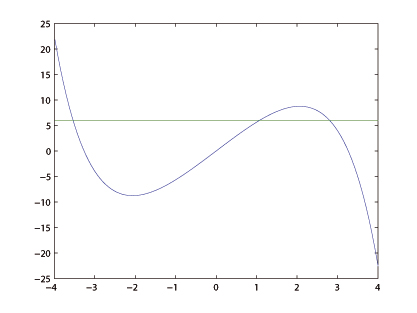}
\caption{\small The graph of  the potential $\Phi$ for the case when {\bf PP} and ${\bf C}$ interactions coexist, $N > 1$}
\end{figure}

\begin{figure}[h!]  \label{Fig3}
\vskip-0.5truecm
 \includegraphics[width=80mm]{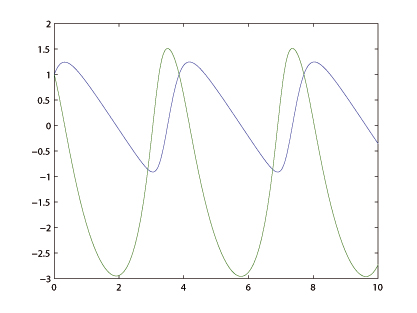}
\caption{\small The plot of  a periodic solution $q(t)$ and $p(t)$}
\end{figure}

Let us formulate  conditions providing  that
$\Phi(q) \to +\infty$ as $|q| \to +\infty$ in the case of arbitrary $a_i$ and $b_i$.
Then  $|q(t)|$ is bounded uniformly in $t$ and hence the population abundances $x_i$
are separated from $0$ and $+\infty$. Therefore, then system  (\ref{LVX}), (\ref{LVV}) is strongly persistent.

({\bf PI}) Assume that all $a_i$ are positive. Let $i_+$ be the index  corresponding to the largest $a_i$.
 Condition   $b_{i_{+}}>0$ and $\bar r >0$ is equivalent to persistency of our system.
  In this case $\Phi(q) \to +\infty$  as $q \to \infty$ and according to (\ref{level}), $|q(t)|$ is bounded;

 ({\bf PII}) Assume that all $a_i$ are negative. Let $i_{-}$ be the index that corresponds to the largest value of $-a_i$.
 If    $b_{i_{-}}< 0$ and $\bar r >0$,  then $|q(t)|$ is bounded and the system is persistent (and vice versa).

 ({\bf PIII}) Let $a_i$  may have different signs. Let $i_{\pm}$ be the indices  corresponding to the maximal values of $\pm a_i$,
 respectively.
  If $b_{i_{+}}>0$ and $b_{i_{-}} < 0$,  then the   system is persistent.

To conclude this section,  let us describe some effects.  First,  condition ({\bf PIII}) shows that there is possible
a {\em domino} effect, when  an extinction of a  species  leads to instability of all species community.

Indeed, let us assume that if  $a_j >0$ for  $j \ne i_{+}$ the coefficient $b_j < 0$.
Then extinction of the $i_{+}$-th species  leads to instability of  the whole species system.

The second effect is a noise-induced transition \cite{Horst}.  Assume that
the potential energy $\Phi(q)$
has a local maximum $\Phi_+$,  and $\Phi(q) \to +\infty $ as $q \to \pm \infty$.  Then $\Phi$ has at least  two local minima
(two potential wells) and, according to ({\bf iii}), a soliton exists. If  the network environment is random,
its fluctuations can generate  random transitions  between the potential wells even if
$E < \Phi_+$.  Such  transitions   provoke  ecological catastrophes.

\subsection{Hamiltonians via $x,v$ and perturbations} \label{hamxv}

	Hamiltonians in variables $(q,p)$ can be represented as functions of
	species abundances $x_i$ and $v$.
Let $m_i$ be positive numbers such that
$\sum_{i=1}^N m_i =1$.
By elementary transformations we obtain that functions
\begin{equation}
   E(x, v, m)=v -\mu \ln(v) + \sum_{i=1}^N \rho_i x_i -\bar r m_i a_i^{-1} \ln(x_i)
    \label{motionintegrals}
\end{equation}
are motion integrals, i.e., conserve on the system trajectories. If we consider
these functions as Hamiltonians, then, in order to write equations in a Hamiltonian
form, we must use a special representation involving a skew-product Poisson matrix depending on species-abundances \cite{Plank}.
We have thus a whole family of motion integrals. There is an interesting property:
if all $\rho_i$ and $a_i$ are positive, then
the minimum of the function  $E(x, v, m)$ gives us an equilibrium $(\bar v, \bar x)$
of the Hamiltonian star system defined by
$$
                   \bar v=\mu, \quad \bar x_i= \bar r m_i (\rho_i a_i)^{-1}.
$$
Different choices of weights $m_i$ correspond to different positive equilibria.
According to Corollary \ref{Corol},  the corresponding Lotka-Volterra system is strongly persistent (but it is not permanent).

Consider the structural  stability of Hamiltonian and non-Hamiltonian star systems
with small concurrence and self-limitation.
Let $M=1$ and $N > 1$. We assume, to simplify calculations, that
$\gamma_{ij}=\gamma_i \delta{ij}$, where $\delta{ij}$ is the Kroneckerdelta.
Let us consider equations  (\ref{LVX}) and (\ref{LVV}),  which can be written down then as follows:
\begin{equation}
     \frac{dx_i}{dt}=x_i (-r_i  +  a_{i} v -   \gamma_{i} x_i),
    \label{LVX11}
     \end{equation}
\begin{equation}
     \frac{dv}{dt}=v (\bar r   -  \sum_{l=1}^N b_{l} x_l  -  d v).
    \label{LVV11}
     \end{equation}
     where $i=1,..., N$.
		Let us denote $\rho_l=b_l/a_l$, $\mu_i= r_i/a_i$ and
		$\theta_i=\rho_i a_i \gamma_i^{-1}$.
		A  positive equilibrium is defined by relations
		\begin{equation}
     \bar x_i =\gamma_{i}^{-1} a_i ( v -\mu_i)=X_i(v),
    \label{Xeq11}
     \end{equation}
\begin{equation}
     \bar v=\frac{\bar r   +  \sum_{i=1}^N \mu_i \theta_i}{d + \sum_{i=1}^N  \theta_i},
    \label{Veq11}
     \end{equation}
provided that $X_i(v) > 0, \ i=1,...,N$.
One can check that for small $\gamma_i, d >0$
this condition holds only for the Hamiltonian case, i.e., when $\mu_i=\mu$ for all
$i$.  This means that the Hamiltonian structure supports diversity.
Consider the total biomass $B=\bar v + \sum_i \bar x_i$ as a function of $\gamma$
in the  case $\mu_i=\mu$ and $\gamma_i=\gamma$ for all $i$.
It is not difficult to check that, if
$\bar r > d\mu$, then $\bar v$ is an increasing function of $\gamma$ and $x_i$ are
decreasing functions of $\gamma$.  For small
$\gamma$ one has
$$
\bar v= \mu + O(\gamma), \quad \bar x_i= a_i (\bar r -  d\mu)(d\gamma+ \sum_{i=1}^N
 \theta_i)^{-1}.
$$
Therefore, for large $N$ the total biomass $B$ decreases in the self-limitation parameter
$\gamma$ and approaches to a maximal limit value as $\gamma \to 0$.

Another interesting fact is that if in the relation for the Hamiltonian
(\ref{motionintegrals}) we put $m_i=\bar x_i(\gamma)$ and $\mu=\bar v$,
 then $E(x, v)$ becomes a Lyapunov function decreasing along trajectories
of the corresponding Lotka-Volterra system.

\section{ Varying environment} \label{pred2}


Consider  system (\ref{LVX}),(\ref{LVV}) assuming that $M=1$ and
that $a_i, b_i,  r_i$ and $\bar r$ can  depend on $t$. The
dependence on $t$  describes  an influence of a varying environment.


System (\ref{LVX}),(\ref{LVV}) takes  the form
\begin{equation}
     \frac{dx_i}{dt}=x_i (-r_i(t)  +  a_{i}(t) v -  \sum_{j=1}^N \gamma_{ij} x_j),
    \label{LVXb}
     \end{equation}
\begin{equation}
     \frac{dv}{dt}=v (\bar r(t)   -  \sum_{j=1}^N b_{j}(t) x_j  -  d_{11} v).
    \label{LVVb}
     \end{equation}
Suppose that there exist constants $\bar \gamma_i >0$ and $\mu$ such that
\begin{equation}
      r_j(t) -   a_{j}(t) \mu= -\bar \gamma_j,  \quad j=1,..., N,
      \label{rr25}
     \end{equation}
 Similar to Section \ref{trans}, we put
$$
x_i=C_i \exp(a_i q), \quad  dq/dt + \mu=v.
$$
After some transformations (compare with Sect. \ref{trans}), we
obtain the following system for $C_i, q$ and $p$:
\begin{equation}
     \frac{dC_i}{dt}=C_i(\bar \gamma_i - \sum_{j=1}^N \gamma_{ij} C_j \exp(a_j q)   -  q \frac{da_{i}}{dt}),
    \label{VXC}
     \end{equation}
 \begin{equation}
     \frac{dq}{dt}=\exp(p)-\mu,
    \label{VVCq}
     \end{equation}
\begin{equation}
     \frac{dp}{dt}= \bar r(t)  -  d_{11} \exp(p) -  \sum_{j=1}^N b_{j} C_j \exp(a_i q),
    \label{VVC}
     \end{equation}
which is equivalent to (\ref{LVXb}) and (\ref{LVVb}). We investigate this system in the next subsection.

\subsection{Weak self-limitation and slowly varying environment}
\label{SWE}
 Let
 the coefficients $a_k, b_k, r_k$ and $\bar r$ be functions of  the slow time $\tau= \epsilon t$, where $\epsilon>0$ is a
 small parameter. We assume that

 ({\bf I}) self-limitations    are small, i.e.,
${d}_{11}=\epsilon \bar d$, where $\bar d > 0$ and ${\gamma}_{ij}=\kappa  \gamma_{i} \delta_{ij}$, where $\delta_{ij}$ stands for the Kroneckerdelta, $ \gamma_i >0$ and $\kappa >0$;

({\bf II})  $\bar \gamma_i =\kappa \hat \gamma_i$, where  $\hat \gamma_i >0$.

Under these assumptions,  system  (\ref{VXC}), (\ref{VVC}) lies in
the class of well studied weakly perturbed Hamiltonian systems \cite{Arnold, Kam}.
  Equations (\ref{VVCq}) and (\ref{VVC})
 take the following form:
\begin{equation}
     \frac{dq}{dt}=\exp(p) -\mu, \quad  \frac{dp}{dt}=f(C(t), q(t), \tau) + \epsilon g( p(t)),
    \label{QHPer}
     \end{equation}
     where  $$
     g=-\bar d \exp(p)
     $$
is a term associated with self-limitation effects for $v$, and
     \begin{equation}
  f(C, q, \tau)=-\sum_{j=1}^N b_j(\tau) C_j \exp(a_j(\tau) q) + \bar r(\tau).
    \label{fQt}
     \end{equation}
   Equation  (\ref{VXC}) for $C_i(t)$ becomes
   \begin{equation}
     \frac{dC_i}{dt}=  \epsilon \tilde W_i(C, q, \tau),
    \label{Cper}
     \end{equation}
   where
   \begin{equation} \label{tWi}
     \tilde W_i(C, q,\tau)=\beta C_i \Big( \hat \gamma_i(\tau)  -  \gamma_{i} C_i \exp(a_i(\tau) q)  -
     q \frac{da_i}{d\tau} \Big)
\end{equation}
   and $\beta=\kappa/\epsilon$.

         System (\ref{QHPer}), (\ref{Cper})   can be resolved by the  averaging method (see \cite{Sari}), which
         gives  rigorous results for time intervals of  order $O(\epsilon^{-1})$.
         According to this method, we can represent $C(t)$ as
         $$
         C(t)=\bar C(\tau) + O(\epsilon),
         $$
         where an equation for $\bar C(\tau)$ will be written below.
         The $q$ and $p$ variables  can be represented in the following multiscale form:
    \begin{equation}
     q=Q(t, \tau, E(\tau)) + O(\epsilon), \quad p=P(t, \tau, E(\tau)) + O(\epsilon),
    \label{persol}
     \end{equation}
        where  the leading terms  $Q$ and $P$  can be found
        from  the system:
        \begin{equation}
     \frac{dQ}{dt}=\exp(P) -\mu, \quad  \frac{dP}{dt}=f(\bar C(\tau), Q,
     \tau),
    \label{QHPer2}
     \end{equation}
		Compare with (\ref{QH}). Here $f$ is given by (\ref{fQ}), where
		$C$, $b_j$ and $\bar r$ depend on the parameter $\tau$.
     This system is Hamiltonian and
          the corresponding energy integral is defined by
        (\ref{energy00}). Taking into account the dependence of $f$ on $\tau$
				in  (\ref{QHPer2}) we write the energy integral as
       \begin{equation}
     \Psi(P)+\Phi(\bar C(\tau), Q, \tau)=E(\tau)
    \label{energyp}
     \end{equation}
for each fixed $\tau$ and $E$. The properties of solutions $Q(t,
\tau, E(\tau))$ can be described as  in  Sect.\ref{pred}. We seek
periodic in $t$ solutions $Q,P$ of system
(\ref{QHPer2}), (\ref{energyp}) assuming that $\tau$ and hence $\bar
C, E$ are parameters.  System (\ref{QHPer2}) should be supplemented
by  the equation describing  behaviour of $E(\tau)$ and $\bar
C(\tau)$ as  functions of $\tau$ (it appears that these equations
are coupled).
 In the multiscale procedure, the equation for $E$  guarantees the boundedness of corrections to the
leading terms $Q$ and $ P$ on the time intervals of length
$O(\epsilon^{-1})$.

The equation for the unknown function $E(\tau)$ can be derived  as follows. Let
$$
 \langle  z(\cdot,\tau) \rangle = T^{-1}\int_0^{T} z(t, \tau) dt
$$
be  the average of a function $z$ over the period $T(E)$.
Using Theorem  3 from \cite{Sari},  one has
\begin{equation}
     \frac{dE}{d\tau}=S_1(E, \bar C, \tau) + S_2(E, \bar C, \tau) + S_3(E, \bar C, \tau),
    \label{energyevol}
     \end{equation}
 where
 $$
   S_1=\langle g(Q(\cdot,\tau, E), P(\cdot, \tau),\tau) (\exp(P(\cdot, \tau, E)) -\mu) \rangle
   $$
   gives the contribution of self-limitation effects,
   $$
   S_2=  \langle \Phi_{\tau} (\bar C(\tau), Q(\cdot, \tau, E), \tau)\rangle
   $$
   is the term determining a direct dependence of $\Phi$ on $\tau$ and
   $$
     S_3=\langle \sum_{i=1}^N \Phi_{\bar C_i}(\bar C(\tau), Q(\cdot, \tau, E), \tau)  W_i(\bar C(\tau), Q(\cdot, \tau, E), \tau)) \rangle
   $$
    is the term coming from  the dependence of $\Phi$ on $C_i$.
Here
\begin{equation} \label{ViDD}
     W_i(\bar C, E, \tau)=\bar C_i (\bar \gamma_i  -   \gamma_{i} \bar C_i \theta_i(\bar C, E)   -  \frac{da_i}{d\tau}
      \langle Q \rangle ),
\end{equation}
where
\begin{equation} \label{theta}
     \theta_i(C, E)= \langle  \exp(a_i Q) \rangle.
\end{equation}
For $\bar C_i$ we have \cite{Sari}
\begin{equation} \label{Cev}
\frac{d\bar C_i}{d\tau}=W_i(\bar C_i, E, \tau).
\end{equation}
Thus, we have obtained  equations (\ref{Cev}) and (\ref{energyevol})
for $\bar C_i(\tau)$  and $E$, respectively.  These equations and
 (\ref{QHPer2}), (\ref{energyp}) give us the complete system that
allows to describe dynamics for time intervals of length
$O(\epsilon^{-1})$.  Note that first two equations
(\ref{QHPer2}), (\ref{energyp}) contain the derivatives with respect
to  $t$, while eqs. (\ref{Cev}) and (\ref{energyevol}) involve the
slow time $\tau$ only.

{\bf Remark }. Hamiltonian $H$ depends on the parameter $\mu$. The averaging allows us to find this parameter.
Note that for periodic solutions relation (\ref{QHPer2}) implies
\begin{equation} \label{muI}
\langle  f(\bar C(\tau), Q, \tau) \rangle= \sum_{k=1}^N b_k \bar C_k \langle \exp(a_k Q) \rangle - \bar r=0.
\end{equation}
It is natural  to choose  $\mu$ in such a way that (\ref{Cev}) has a stationary solution.  This solution is defined
by
\begin{equation} \label{muII}
 \bar C_k = \bar \gamma_k ( \gamma_k \langle \exp(a_k Q) \rangle)^{-1}.
\end{equation}
Substituting (\ref{muII}) into (\ref{muI})  and using the definition of $\bar \gamma_k$,
one has the following relation
\begin{equation} \label{muIII}
\sum_{k=1}^N  b_k (r_k - a_k \mu)\gamma_k^{-1}=0.
\end{equation}
One can verify  that $\mu=\bar v >0$, where  $\bar v$ is the $v$-component of an equilibrium solution $(\bar x, \bar v)$ of  system (\ref{LVX}), (\ref{LVV})
in the case  $M=1$, ${\bf D}={\bf 0}$ and $\bf \Gamma$ is a diagonal matrix.

One can show that solutions of this ''averaged" system
(\ref{QHPer2}), (\ref{energyp}), (\ref{energyevol}), (\ref{Cev}) are
close to the solutions of the original system (\ref{QHPer}),
(\ref{Cper}) on a time interval of  length  $O(\epsilon^{-1})$.
Equation (\ref{energyevol})  expresses an ''averaged" energetic
balance:
  three factors define evolution of averaged energy, namely,  the evolution of $\bar C$,  the dependence of the parameters
  on $\tau$ and  self-limitation effects.

\subsection{Main effects: irregular bursts, quasiperiodic solutions and chaos} \label{effects}

To proceed with a qualitative  analysis of solutions to system
(\ref{QHPer}), (\ref{Cper}), we assume that
\begin{equation} \label{Phias}
\Phi(C, q) \to +\infty   \quad \mbox{as} \ q \to \pm \infty.
\end{equation}
Validity of this condition is analyzed in Sect.\ref{globs}, and (\ref{Phias})
holds for random $a_j$ and $b_j$ if $a_j/b_j=\rho_j >0$. Let
$q_j^*(C)$ be local extrema of $\Phi$ for a given $C$. Let us put $\Phi_j=\Phi(q_j^*(C))$.

Energetic relation (\ref{energyevol}) allows us to show existence of interesting phenomena. In fact,  if $\Phi_j$ is a local maximum, then
evolution of $E$ can lead to special periodic solutions in $q$, when $E$ approaches  the value $\Phi_j$. These
solutions can be represented as periodic sequences of bursts separated by  large time intervals $T_0(E, \tau)$ (see Fig. \ref{Fig4}). If $E=\Phi_j$,
we have a single burst (soliton) and $T_0=\infty$.  Note that Fig. \ref{Fig4} illustrates the case
of fixed $\tau$ and $E$ ($\epsilon=0$). For $E$ close to $\Phi_j$  and small $\epsilon >0$ the time behaviour of $q(t, \tau)$  exhibits a
chain of slightly different bursts separated by different large time intervals (see Fig. \ref{Fig5}). Existence of solitons means that
there is a homoclinic structure in the unperturbed Hamiltonian dynamics, and therefore, this sequence of bursts can be chaotic as a
result of $\tau$-evolution \cite{Arnold, Robinson}.

\begin{figure}[h!] \label{Fig4}
\vskip-0.5truecm
 \includegraphics[width=80mm]{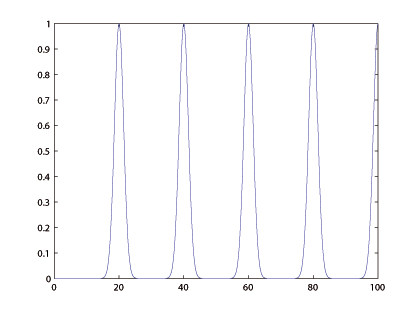}
\caption{\small The plot of  a periodic solution $q(t)$ with a large period $T$ for $\epsilon=0$}
\end{figure}

\begin{figure}[h!]   \label{Fig5}
\vskip-0.5truecm
 \includegraphics[width=80mm]{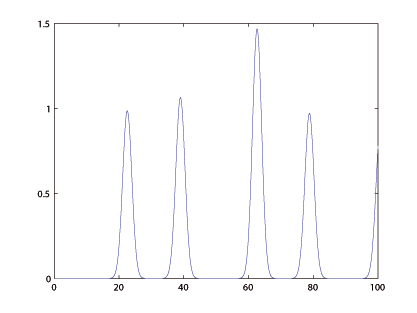}
\caption{\small The plot of  a  solution $q(t)$  close to a sequence of the bursts for small $\epsilon$}
\end{figure}

The following picture  of the time evolution of solutions $q(t)$ can be observed.  For  $E$ close to a local minimum of $\Phi$ we have
periodic oscillations with an amplitude and a period, which slowly evolve in $t$.
  When $E$ approaches  a local maximum of $\Phi$,
 we obtain a irregular chain of rare bursts.
Such a picture is observed in macroscopic ecological dynamics  (see  \cite{Uyeda}).

The following effects can occur here:

({\bf Ai})   Let $E(0)  > \Phi_l^*$ for some $l$. The value $E(\tau)$ does not meet  values $\Phi_l^*$ for all $\tau$.
Then we deal with only periodical solutions with a period and an amplitude depending on $\tau$. This
means that the environment stabilizes the population against self-limitation;

({\bf Aii}) The value $E(\tau) $  passes through  $\Phi_l^*$ for some $\tau$, then we have an ecological burst;

({\bf Aiii})  An ecological burst is also possible when $E(0) < \Phi_l^*$ for all $l$ and $S_2$ is more than $|S_1|$.
Then we observe that the environment destabilizes the population against self-limitation.

So,  the climate oscillations can stabilize ecological dynamics in certain cases.

Consider more complicated situations.  A system of equations similar to (\ref{Cev}) and (\ref{energyevol}) can be derived (at least,
formally) in the general multidimensional case $M >1$ if the parameters $\bar \gamma_i$ (defined by (\ref{GB})) are small, and the
potential energy $\Phi(q)$ satisfies conditions (\ref{stab25}). Indeed, the behaviour of solutions of non-perturbed
Hamiltonian system is defined by the energy $E$. If $E$ is close to a local minimum of $\Phi$, then, at least for some values of $E$, we
have quasiperiodic solutions that follows from the KAM theory (see \cite{Arnold, Kam}). Then the averaging procedure leads to a system
analogous to  (\ref{Cev}) and (\ref{energyevol}). In the  multidimensional case effects ({\bf Ai})- ({\bf Aiii}) are also possible if the
potential energy $\Phi$ has local minima and saddle points. All these effects are induced by non-Hamiltonian perturbations.

Other interesting situations appear for an ecological system
decomposed into $n>1$ weakly interacting compartments, which are
star systems with $M=1$ and $N > 1$. Let us assume that
self-limitation is absent: $\bar \gamma_i=0$, $\gamma_i=0$ and ${\bf
D}=0$. In this case the Hamiltonian $H(C, q, p)$ can be represented as
\begin{equation} \label{KAM1}
H=\sum_{l=1}^n \Psi(C,  p_l) + \Phi(C, q_l) + \kappa \tilde \Phi(C, q),
\end{equation}
It consists of two independent terms and  a small
contribution $\kappa \tilde \Phi$ describing a weak interactions
between compartments. For example, such  situation can occur if we
have $n$  preys and $Nn$ of predators. Each predators usually eats
some special types of prey, but sometimes (with a small frequency
$\kappa >0$) different predators share the same prey.

The Hamiltonian systems with Hamiltonians (\ref{KAM1}) are studied in the KAM theory \cite{Arnold, Kam}.
For small  $\kappa$ the most of trajectories are  quasiperiodic with
slowly evolving time periods and amplitudes. A small part of trajectories can exhibit a chaotic behaviour.

\section{Resonances}\label{Reson}

\subsection{Two interacting star systems}

   The resonance  analysis is important in the Hamiltonian dynamics investigation,
 since  resonances can lead to instabilities, periodical oscillations, chaos,  and other
 interesting effects in systems with many variables. These effects are important for mechanical and physical applications.  However,  resonances have not considered yet
 for large ecological webs.  For example, work \cite{King} considered
the case of two species  predator-prey systems perturbed  by small time periodic climate
variations. In opposite to  \cite{King}, we consider internal resonances, when
there are no external variations and the resonance effect is generated by system interactions.  We show that in ecological networks such internal
resonance effects exist and can provoke  instabilities.

 Let us consider two star subsystems, the first one involves variables $v$ and $x_i$,  $i=1,..., N_1$,
 while  the second subsystem involves abundances $w$ and $y_i$,  $i=1,...,N_2$.  System of equations describing a weak interaction between
 these subsystems can be written in the following form:
 \begin{equation}  \label{xweak1}
 \frac{dx_i}{dt}=x_i(- r_i^{(1)} +  a_i^{(1)} v + \kappa \tilde a_i^{(1)} w - \gamma_i^{(1)} x_i),
 \end{equation}
  \begin{equation}  \label{vweak1}
\frac{dv}{dt}=v(\bar r_i^{(1)} -  \sum_{i=1}^{N_1} b_i^{(1)} x_i  - \kappa  \sum_{j=1}^{N_2} \tilde b_j^{(1)} y_j - \epsilon d_1 v),
  \end{equation}
 \begin{equation}  \label{xweak2}
 \frac{dy_j}{dt}=y_j(- r_j^{(2)} +  a_i^{(2)} v + \kappa \tilde a_i^{(2)} w - \gamma_j^{(2)} y_j),
 \end{equation}
  \begin{equation}  \label{vweak2}
 \frac{dw}{dt}=w(\bar r_i^{(2)} -  \sum_{j=1}^{N_2} b_j^{(2)} y_j - \epsilon  \sum_{i=1}^{N_1} \tilde b_i^{(2)} x_i - \epsilon d_2 w).
 \end{equation}
 Here $\kappa >0$ and $\epsilon>0$ are small parameters such that $\kappa >> \epsilon$.  The terms proportional to $\kappa$ describe
 a weak interaction of two subpopulations with star structures.  The terms proportional to $\epsilon$  correspond to
 self-limitation effects.
We  assume that
$$
r_i^{(k)}-\mu_k a_i^{(k)}=-\epsilon \tilde \gamma^({k}), \quad k=1,2,
$$
and introduce variables $q_1$ and $ q_2$   (compare with
Sect. \ref{trans}) by
$$
\frac{dq_1}{dt}+ \mu_1=v, \quad \frac{dq_2}{dt}+ \mu_2=w,
$$
and
\begin{equation} \label{Ham1}
x_i = C_i^{(1)} \exp(a_i^{(1)} q_1), \quad  y_j = C_j^{(2)} \exp(a_i^{(2)} q_2),
\end{equation}
where $C_i^{k}$  are unknowns. Let us define $p_i$ by
\begin{equation} \label{Ham1}
\frac{dq_i}{dt}=\exp(p_i)- \mu_i,   \quad i=1,2.
\end{equation}
Then  $p_1$ and $p_2$ satisfy
\begin{equation} \label{Hamp1}
\frac{dp_1}{dt}= -\Phi_q^{(1)}(q_1) + \kappa g_1(q_2) - \epsilon d_1 \exp(p_1),
\end{equation}
\begin{equation} \label{Hamp2}
\frac{dp_2}{dt}= -\Phi_q^{(2)}(q_2) + \kappa g_2(q_1) - \epsilon d_2 \exp(p_2),
\end{equation}
where
$$
\Phi^{(1)}(q_1)= \sum_{i=1}^{N_1} \frac{C_i^{(1)} b_i^{(1)}}{ a_i^{(1)}}  \exp(a_i^{(1)} q_1) - \bar r^{(1)}q_1,
$$
$$
\Phi^{(2)}(q_2)= \sum_{i=1}^{N_2} \frac{C_i^{(2)} b_i^{(2)}}{ a_i^{(2)}}  \exp(a_i^{(1)} q_2) - \bar r^{(2)} q_2,
$$
$$
g_1(q_2)=\sum_{j=1}^{N_2} \tilde b_j^{(1)}  C_j^{(2)} \exp(a_j^{(2)} q_2),
\quad g_2(q_1)=\sum_{i=1}^{N_1} \tilde b_i^{(2)}  C_i^{(1)} \exp(a_i^{(1)} q_1).
$$

For  fixed $C_i^{k}$  we  obtain the weakly perturbed Hamiltonian
system, defined by equations (\ref{Ham1}),(\ref{Hamp1}) and
(\ref{Hamp2}). We suppose that for $\kappa=\epsilon=0$ this system
has  an equilibrium solution
\begin{equation} \label{eqstate}
q_1(t) \equiv \bar q_1, \quad q_2(t) \equiv \bar q_2
\end{equation}
and periodical solutions oscillating around this equilibrium.

\subsection{Asymptotic analysis}

To simplify the statement, we consider the case of small periodic
oscillations near equilibrium (\ref{eqstate}). Then we keep only
 quadratic terms in the Taylor expansion of $\Phi^k$, i.e.
$$
\Phi^{1}(q_1)=\omega_1^2 \tilde q_1^2,  \quad \Phi^{2}(q_2)=\omega_2^2 \tilde q_2^2,  \quad  \tilde q_i=q_i - \bar q_i.
$$
The functions $g_i$ can be approximated by linear terms as follows:
$$
g_1(q_2)=g_1  +  g_{12} \tilde q_2 + O(\tilde q_2^2),  \quad   g_2(q_1)=g_2  +  g_{21} \tilde q_1 + O(\tilde q_1^2),
$$
where
$$
g_{12}= \frac{dg_1(q)}{dq}(\bar q_1), \quad g_{21}=\frac{dg_2(q)}{dq}(\bar q_2).
$$
In the case of  small oscillations system (\ref{Ham1}),(\ref{Hamp1})
and  (\ref{Hamp2}) can be written as a linear system of  second
order
\begin{equation} \label{Hamp1S}
\frac{d^2\tilde q_1}{dt^2}  + \omega_1^2 \tilde q_1= \kappa g_{12} \tilde q_2 - \epsilon d_1 (\frac{\tilde dq_1}{dt} +\mu_1),
\end{equation}
\begin{equation} \label{Hamp2S}
\frac{d^2 \tilde q_2}{dt^2} + \omega_2^2 \tilde q_2= \kappa g_{21} \tilde q_1 - \epsilon d_2 (\frac{\tilde dq_2}{dt} +\mu_2).
\end{equation}
The resonance case occurs if
$$
\omega_1=\omega_2=\omega.
$$
If $|\omega_1 - \omega_2| >> \kappa$, then  system
(\ref{Hamp1S}),(\ref{Hamp2S})  can be resolved in a simple way and
the solutions are
small regular perturbations of periodic limit cycles. Let us
consider the resonance case.

To resolve  system  (\ref{Hamp1S}), (\ref{Hamp2S}), we apply
a standard asymptotic method. Let us introduce a slow time
$\tau=\kappa  t$. We are looking for asymptotic solutions in the form
\begin{equation} \label{Hamp2Sq1}
\tilde q_k=  Q_k(\tau) \sin(\omega t  + \phi_k(\tau)) + \kappa S_k(t, \tau) + ...  , \quad k=1, 2
\end{equation}
where $Q_k$ and $\phi_k$ are new unknown functions of $\tau$, $S_k(t, \tau)$  are corrections of the main terms.  Here
$Q_k$ define  slowly evolving amplitudes of the oscillations whereas $\phi_k$ describe phase shifts. Differentiating  (\ref{Hamp2Sq1}) with respect to $t$
and
substituting the relations obtained into (\ref{Hamp1S}) and (\ref{Hamp2S}), we have
\begin{equation} \label{Hamp3S}
\frac{\partial ^2 S_k}{\partial t^2}  + \omega^2 S_k= F_k(Q_1, Q_2, \phi_1, \phi_2, t),
\end{equation}
where
\begin{eqnarray*} 
 F_1=-(2\omega\frac{dQ_1}{d\tau}+\bar\epsilon d_1 \omega Q_1)  \cos(\omega t  + \phi_1) +
2\omega Q_1 \frac{d\phi_1}{d\tau} \sin(\omega t  + \phi_1) + \\  +  g_{12} Q_2 \sin(\omega t +\phi_2)  + \tilde \mu_1) + O(\kappa^2),
\end{eqnarray*}
\begin{eqnarray*} 
 F_2=-(2\omega\frac{dQ_2}{d\tau} +\bar\epsilon d_2 \omega Q_2 ) \cos(\omega t  + \phi_2) +
 +2\omega Q_k \frac{d\phi_2}{d\tau} \sin(\omega t  + \phi_2)  + \\ +
g_{21} Q_1 \sin(\omega t +\phi_1)  +\tilde \mu_2) + O(\kappa^2).
\end{eqnarray*}
Here $\bar \epsilon=\epsilon/\kappa$ and $\tilde \mu_k=\bar \epsilon
\mu_k d_k$.

We seek solutions $S_k$ of  (\ref{Hamp3S}), which are $O(1)$
as $\kappa \to 0$. Such solutions exist if and only if the
following relations hold:
\begin{equation} \label{Hamp6S}
\int_0^T F_k(Q_1, Q_2, \phi_1, \phi_2, t) \sin(\omega t+\phi_k)dt=0,
\end{equation}
\begin{equation} \label{Hamp7S}
\int_0^T F_k(Q_1, Q_2, \phi_1, \phi_2, t) \cos(\omega t+\phi_k)dt=0,
\end{equation}
where $T=2\pi/\omega$.
Evaluation of the integrals in (\ref{Hamp6S}), (\ref{Hamp7S}) gives
the following system
  for the amplitudes $Q_k$ and the phases $\phi_k$:
\begin{equation} \label{HampQ1}
 \omega\frac{dQ_1}{d\tau}=-\bar\epsilon d_1 \omega Q_1 + b_{12} Q_2  \sin(\phi_2-\phi_1),
\end{equation}
\begin{equation} \label{HampQ2}
 \omega\frac{dQ_2}{d\tau}=-\bar\epsilon d_2 \omega Q_2 + b_{21} Q_1  \sin(\phi_2-\phi_1),
\end{equation}
\begin{equation} \label{Hampf1}
 \omega Q_1\frac{d\phi_1}{d\tau}= -b_{12} Q_2  \cos(\phi_2-\phi_1),
\end{equation}
\begin{equation} \label{Hampf2}
 \omega Q_2\frac{d\phi_2}{d\tau}= b_{21} Q_1  \cos(\phi_2-\phi_1),
\end{equation}
where $ b_{12}= g_{12}/2,  \quad  b_{21}= -g_{21}/2. $ We refer to
these equations as {\em resonance} system.

\subsection{Investigation of the resonance system}

The resonance system  can be studied
analytically in some cases. Let $\phi_2(0)-\phi_1(0)=(2n+1)\pi/2$,
where $n$ is an integer.  Then equations (\ref{Hampf1}),
(\ref{Hampf2}) show that $\phi_2(\tau)-\phi_1(\tau)=(2n+1)\pi/2$ for
all $\tau \ge 0$ and thus $\sin(\phi_2(\tau)- \phi_1(\tau))=\pm 1$.
As a result, we reduce (\ref{HampQ1}), (\ref{HampQ2})  to the linear
system
\begin{equation} \label{HampQ1L}
 2\omega\frac{dQ_1}{d\tau}=-\bar\epsilon d_1 \omega Q_1 \pm b_{12} Q_2,
\end{equation}
\begin{equation} \label{HampQ2L}
 2\omega\frac{dQ_2}{d\tau}=-\bar\epsilon d_2 \omega Q_2 \pm b_{21} Q_1.
\end{equation}
If $\bar \epsilon >> 1$, i.e., the self-limitation is stronger than
the interaction, then solutions of this system are exponentially
decreasing and we have  stability. If $\bar \epsilon << 1$, then
solutions of this system are exponentially increasing and we have
instability under condition $ b_{21} b_{12} >0$, i.e.
\begin{equation} \label{Inscond}
  R=\frac{dg_1(q)}{dq}(\bar q_1)\frac{dg_2(q)}{dq}(\bar q_2) <0.
\end{equation}
We see that if  $a_i^{(k)}, \tilde b_i^{(k)} >0$ for all $i$, then $R
>0$ and we have a stable dynamics ($Q_i$ are exponentially
decreasing).

This relation leads to the following biological conclusion.
Instability occurs as a result of resonances only  if prey-predator
interactions are mixed with  other ones, which perturb prey-predator
system (even if these perturbations are small).

\section{Populations of random structure}
\label{random}

The Hamiltonian approach proposed above allows us  to show why complexity can lead to stability and gives a simple method to estimate the
number of unstable and stable web equilibria.

Let us start with one-dimensional case $q \in {\mathbb R}$. Consider the potential energy without
non-essential linear terms:
 \begin{equation} \label{rH}
\Phi(q)=\sum_{k=1}^N b_k \exp(a_k q), \quad b_k=\rho_k C_k.
\end{equation}
If $\Phi(q)$ has a unique minimum and $\Phi(q) \to +\infty$ as $Q \to \pm \infty$, we obtain stability as it was discussed above. Assume
that $b_k$ are random mutually independent coefficients such that the averages $E b_k =\bar b >0$ and the deviation is $\sigma_b$.
Furthermore, we assume that $a_k$ are random independent coefficients such that $Ea_k=0$ and $\sigma(a_k)=\sigma_a$. Then for large $N$ we obtain
the following typical plot like a parabolic curve illustrated by Fig. 4.

\begin{figure}[h!]  \label{Pot}
\vskip-0.5truecm
 \includegraphics[width=80mm]{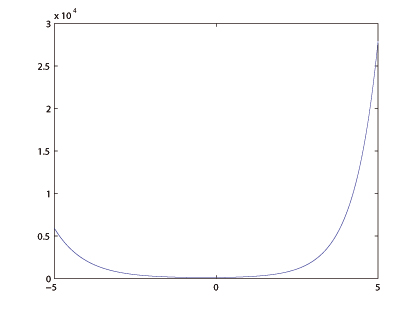}
\caption{\small The plot of  the potential $\Phi(Q)$ for a random Hamiltonian with $N=100$, $\bar b=1$, $\sigma_a=2$, $\sigma_b=10$}
\end{figure}


Numerical results show that stability  increases when the network size $N$ increases. For  $\bar b=1$, $\sigma_b=10$ and $\sigma_a=5 $  numerical
simulations  show that among random Hamiltonian networks with $N \in (1, 100)$  there are  $20$ networks exhibiting instability (the
plot is not parabolic), and  for $N \in (500, 1000)$ only  $4$   networks demonstrate instability.

The next plots show the probability to have a star-structure with time periodic solutions  and with periodic solutions and solitons.  We have studied
a dependence on the number of species $N$.  The probabilities are computed by $150$ test examples with random parameters.

\begin{figure}[h!] \label{sol}
\vskip-0.5truecm
 \includegraphics[width=80mm]{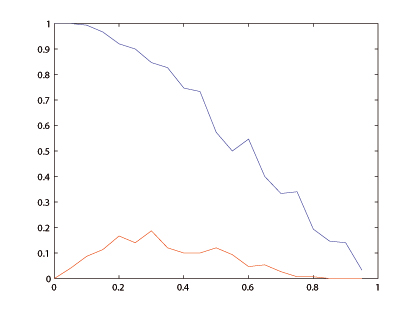}
\caption{\small The plot of  the probability to have a periodic solution (the green curve) and soliton (the red curve).  The horizontal axis:
the probability that an interaction is not predator-prey (mutualism or competition). Here $N=10$, $\bar |a|=\bar |b|=1$, $\sigma
b=\sigma_a=0.5$ and $\bar r=5$.}
\end{figure}

\begin{figure}[h!]   \label{solmutu}
\vskip-0.5truecm
 \includegraphics[width=80mm]{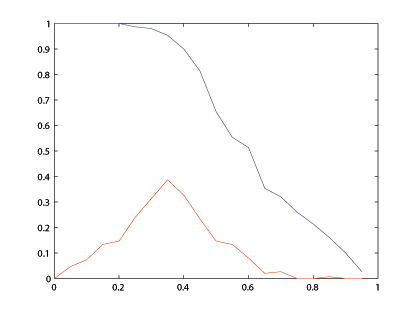}
\caption{\small The plot of  the probability to have a periodic solution (the green curve) and soliton (the red curve).  The horizontal axis:
the probability that an interaction is not predator-prey (mutualism or competition). Here $N=40$, $\bar |a|=\bar |b|=1$, $\sigma
b=\sigma_a=0.5$ and $\bar r=5$.}
\end{figure}


\section{Conclusion}

In this paper, we  apply  Hamiltonian and perturbation methods  to
large complex ecological webs in order to describe
global stability,  and complicated dynamic phenomena.
These methods exploit standard, very effective in mechanics
and physics, canonical variables. Such approach, which is new for ecology,
allows to obtain a new short description of large systems.

For  Hamiltonian Lotka -Volterra systems (HLV)
 sufficient and necessary conditions of persistence (global ecological stability)
are found. It is shown that random large HLV systems are more stable:
stability of a large predator prey system increases in system size, if the
predator species use an adaptive strategy.

Results describing
transient dynamics and complex a dynamical phenomena are presented.
The Hamiltonian
 gives a short system description
that allows us to describe how varying environment (for example, climate), a weak self-limitation and concurrence affect ecological systems.
 Note that macroscopic ecological dynamics was analyzed by a great base of experimental data in
 \cite{Uyeda}.
 The best-fitting
model  explaining these experimentally observed time patterns is a model that combines rare but substantial bursts  with bounded fluctuations on shorter timescales.
Such behaviour can be explained by our model. Namely,
 a weak self-limitation may lead not only to  a simple oscillation damping but also  to  more interesting effects. By averaging methods, we obtain that a typical
 dynamics
 of a system,  where different interactions coexist, is as follows.  There occur
 periodic oscillations. The period and the form of these oscillations slowly evolve in time. After a long time evolution,
 these periodic oscillations can be transformed to a chaotic chain of rare bursts. Mathematically, these bursts are connected with some special solutions like
 solitons and kinks  in physics.
 A slowly varying environment   can also provoke a chaos,   but in some cases it can repress  bursts and  stabilizes dynamics. So, the climate may increase ecological system stability.

The methods, used in this paper, take into account   the topological structure of large networks, which were  studied  during last decades (for example,
\cite{AB, Bas1, Bas2} and references therein). It allows us to investigate the dynamical behavior of complicated ecological networks in more detail.
In fact, our approach is capable to study the network dynamics globally, not only in a local neighborhood of equilibria that allowed to develop
 results of \cite{May1, May2, Alles, Alles1, Alles2}, to describe possible
transient dynamics and finally
reveals some mechanisms of large ecosystem destruction and evolution.
 It is shown that solitons, kinks, and resonances  are possible for
large ecosystems. Such effects
 can be connected with ecological catastrophes. In particular, it is
quite possible that large ecological system can collapse without any explicit
external causes as a result of an internal resonance.
It is shown that these effects arise as a result of existence of mixed
interactions (say, predator-prey plus a weak competition) and niche overlapping.

Finally, we conclude that a combination of  the Hamiltonian interaction structure  with
weak concurrence and self-limitation leads to a complete description of
dynamical phenomena  for large systems.  We describe
globally stable food webs having a large biomass, biodiversity and exhibiting complicated dynamical effects.
Moreover,  let us note that,  first in literature,
we successfully use rigorous mathematical methods to investigate
such large complex ecological networks.

{\bf Acknowledgements}

This second author  was financially supported by
Linkoping University,
Government of Russian Federation, Grant 074-U01 and by grant 13-01-00405-а of Russian Fund of Basic Research.
Also he was also supported in part by grant
RO1 OD010936 (formerly RR07801) from the US NIH.


\end{document}